\documentclass[a4paper,12pt]{article} 

\usepackage{amsfonts,amssymb,amsthm,amsmath}
 \usepackage[dvips]{graphics}
\usepackage{subfigure}
 \newcounter{idesc} 

\DeclareSymbolFont{AMSb}{U}{msb}{m}{n}

 \def\P2m{\P^2(M)}

\begin{document}


\def\N{\mathbb N}
\def\M{\mathbb M}
\def\fdp{f\circ d_{p}}
\def\R{\mathbb R}
\def\DE{D(\mathcal E)}
\def\DEC{\mbox{D}_c(\mathcal E)}
\def\mint{\,\mathop{\makebox[0pt][c]{$\int$}\makebox[0pt][c]{--}\,}}
\def\mintl{\!\!\mathop{\makebox[0pt][c]{$\int$}\makebox[0pt][c]{--}\!}\limits}
\def\suml{\sum\limits}
\def\Er{\mathcal E^r}
\def\E{\mathcal E}
\def\Erw{\mathbb E}
\def\intl{\int\limits}
\newcommand{\Br}[1]{B_r(#1)}
\newcommand{\Bbr}[1]{\mathbb B_r(#1)}
\newcommand{\bnk}[1]{b_{n,k}(#1)}
\newcommand{\bn}[1]{b_{n}(#1)}
\def\phx{\phi_x}
\def\Mkx{\M_k(x)}
\def\bm{{\mathop{\mbox{\normalfont \,vol}}}}
\def\dbm{{{\mbox{\normalfont \tiny d \!\!vol}}}}
\def\dbmn{{{\mbox{\normalfont d \!\!vol}}}}
\def\Qrx{q_r(x)}
\def\Regset{X\setminus S_X}
\def\volsing{\mathbb S_X}
\def\angle{\sphericalangle}
\def\Eomega{\E_\Omega}
\def\Deomega{\mathcal D(\Eomega)}
\def\Meas{\cal M}
\def\1{\,{\makebox[0pt][c]{\normalfont    1}
\makebox[2.5pt][c]{\raisebox{3.5pt}{\tiny {$\|$}}}
\makebox[-2.5pt][c]{\raisebox{1.7pt}{\tiny {$\|$}}}
\makebox[2.5pt][c]{} }}
\def\one{\1 }
\def\eI{[0,1]}

\def\one{\1 }
\def\B{\mathcal B}
\newcommand{\infl}[1]{\inf\limits_{#1}}
\newcommand{\gsk}[1]{\left\{ #1 \right\}}
\newcommand{\norm}[1]{\left\| #1 \right\|}
\def\P{{ /\!\!/}}
\def\ul{\underline}
\def\H{\mathcal H}
\def\for{\mbox{ for }}
\def\Ricc{\mathop{\mbox{\normalfont Ricc}}}
\def\Sec{\mathop{\mbox{Sec}}}

\def\sth{\,|\,}
\def\Hess{\mathop{\mbox{{\normalfont Hess}}}}
\def\einh{\frac{1}{2}}
\def\ol{\overline}
\def\abs{~\\~\\}
\def\sabs{~ \smallskip ~}
\def\k{,\,}
\def\supp{\mathop{\mbox{supp}}}
\def\Curv{\mbox{Curv}}

\def\cut{\mbox{\normalfont Cut}}
\def\mt{\rightarrow}
\newcommand{\lgauss}[1]{{\lfloor #1 \rfloor}}
\def\falls{\mbox{ if }}
\def\sonst{\mbox{ else}}
\def\grad{\mbox{grad}}
\def\div{{{\mathop{\,{\rm div}}}}}
\def\diam{{{\mathop{\,{\rm diam}}}}}
\def\Id{{\mathop{{\bf 1}_{\small \R^d}}}}
\def\k{,\,}
\def\dist{{\mbox{dist}}}
\def\Ind{{\mathop{\mbox{I}}}}
\def\dil{{\mathop{\mbox{dil}}}}
\def\H{\mathbb H}
\def\mtm{M \times M}
\def\BV{\mbox{\normalfont BV}}
\def\gij{{g_{ij}}}
\def\gije{{g_{ij}^\epsilon}}
\def\Lip{\mathop{\mbox{\normalfont Lip}}}
\def\lipc{\mathop{\mbox{\textit {lip}}}}
\def\Lipc{\mathop{\mbox{\textit {Lip}}}}
 \def\loc{{\mbox{\scriptsize loc}}}
\renewcommand{\bullet}{{\mathbf \cdot  }}
\def\Cut{\mbox{\normalfont Cut}}
\def\dist{\mbox{\normalfont dist}}
\def\sign{\mbox{\normalfont sign}}
\def\pr{\mbox{\normalfont pr}}

\def\ioe{\frac{1}{\epsilon}}
\def\fe{f_{\epsilon}}
\def\Ebr{\E^{b,r}}
\def\Dco{D_c(\E_\Omega)}

\newcommand{\D}{{\mathbb{D}}}
\newcommand{\T}{{\mathbb{T}}}
\newcommand{\DD}{\overline{\mathbb{D}}}
\newcommand{\LL}{{\mathbb{L}}}
\newcommand{\EE}{{\mathbb{E}}}

\newenvironment{bew}{\begin{proof}}
{\end{proof}}


%
%
\newtheorem{thm}{Theorem}
\newtheorem{rthm}{Theorem}
\newtheorem{satz}{Proposition}
\newtheorem{lem}[]{Lemma}
\newtheorem{cor}[]{Corollary}
\newtheorem{defn}[]{Definition}

\theoremstyle{definition}
\newtheorem{bem}[]{Remark}
\newtheorem{bemn}[bem]{Remarks}
\newtheorem{bsp}[]{Example}
\newtheorem{bspe}{Examples}

\newcommand\iitem{\stepcounter{idesc}\item[(\roman{idesc})]}
\newenvironment{idesc}{\setcounter{idesc}{0}
\begin{description}} {\end{description}}

\def\dnvol{{d \ol{\normalfont v}}}

\def\P{{\mathcal P}} 
\def\Z{{\mathcal Z}} 
\def\CMS{CM[0,1]^2}
\def\eI{{[0,1]}}
\def\L2{{L^2}}
\def\L11{\mathcal P _{ac}}
\def\S1{{S^1}}
\def\llangle{\langle \langle}
\def\rrangle{\rangle \rangle}

\newcounter{zaehler}
\setcounter{zaehler}{1}
\newcommand{\inszahl}{{\normalfont\tiny\thezaehler.}\addtocounter{zaehler}{1}}

\def\Pbeta{\mathbb P^\Beta}
\def\G{\mathcal G}
\def\Qbeta{{\mathbb Q^\beta}}
\def\Qnull{{\mathbb Q^0}}
\def\dreh{\mathcal R}
\def\Pbeta{\mathbb P^\beta}
\def\ssurd{{\hfill $\surd$}}
\def\bbox{{\hfill $\Box$}}
\def\Def{{\bf  \inszahl Definition.~}}
\def\Satz{{\bf \inszahl Satz. ~}}
\def\Lem{{\bf \inszahl Lemma.~}}
\def\Bew{{\it Beweis. ~}}
\def\Kor{{\bf \inszahl Korollar. ~}}
\def\Top{~\\{\bf \inszahl\,}}
\def\id{{\mbox{id}}}
\def\L{{\mathcal L}}
\def\Lbeta{{\mathcal L^\beta}}

\def\intfgamma{{\mbox{$\int\!F(\gamma)$}}}
\def\intfg{{\mbox{$\int\!F(g)$}}}
\def\intggamma{{\mbox{$\int\!G(\gamma)$}}}
\def\inthg{{\mbox{$\int\!H(g)$}}}
\def\nN{\mathbb N}
\def\H{{\mathcal H}}
\def\C{{\mathcal C}}
\def\D{{\mathcal D}}
\def\smint{{\mbox{$\int$}}}
\def\sabs{ \\ \smallskip}
\def\inthg{{\mbox{$\int\!H(g)$}}}
\def\rinv{{\rm rinv}}
\def\remarkend{}
\def\dnull{{d_\E}}

\newcommand{\leb}{{\mbox{Leb}}}
\newcommand{\Cyl}{{\mathfrak{C}}}
\newcommand{\Syl}{{\mathfrak{S}}}
\newcommand{\Zyl}{{\mathfrak{Z}}}
\def\smint{{\mbox{$\int$}}}

\newcommand{\MCP}{\mbox{{\sf{MCP}}}}
\newcommand{\CD}{\mbox{{\sf{CD}}}}
\newcommand{\Length}{\mbox{\rm Length}}
\newcommand{\Ent}{\mbox{\rm Ent}}
\newcommand{\Gaps}{\mbox{\rm gaps}}

\newcommand{\Dom}{{\it{Dom}}}
\newcommand{\tr}{\mbox{\rm tr}}
\newcommand{\X}{{\mathbb{X}}}
\newcommand{\A}{{\mathcal{A}}}
\newcommand{\U}{{\mathcal{U}}}
\newcommand{\V}{{\mathcal{V}}}
\newcommand{\inv}{\mbox{\sf inv}}
\newcommand{\cum}{\mbox{\sf cum}}
\newcommand{\varh}{{h}}
\newcommand{\Pz}{{\mathcal{P}_2}}
\newcommand{\Pe}{{\mathcal{P}}}
\newcommand{\riccurv}{\,\underline{\mathbb{C}\mbox{\rm urv}}}
\newcommand{\alexcurv}{\,\underline{\mbox{{\sf curv}}}}

\newtheorem{theorem}{Theorem}[section]
\newtheorem{lemma}[theorem]{Lemma}
\newtheorem{definition}[theorem]{Definition}
\newtheorem{corollary}[theorem]{Corollary}
\newtheorem{proposition}[theorem]{Proposition}
\newtheorem{remark}[theorem]{Remark}
\newtheorem{example}[theorem]{Example}

\newcommand{\Q}{{\mathbb{Q}}}
\newcommand{\Pp}{{\mathbb{P}}}
 
\renewcommand{\labelenumi}{(\roman{enumi})}

\def\proof{\smallskip \noindent{\textit{Proof.~}}}

\title{{\Large  Particle Approximation of the Wasserstein Diffusion\\
}}
\author{ Sebastian Andres and Max-K. von Renesse
\def\thefootnote{} \footnote{Technische Universit\"at Berlin, email: [andres,mrenesse]\@@math.tu-berlin.de}
}

 \maketitle \abstract{We construct a system of interacting two-sided Bessel processes on the unit interval and show that the  associated empirical measure process converges to the Wasserstein Diffusion \cite{vrs07}, assuming that Markov uniqueness holds for the generating Wasserstein Dirichlet form. 
The proof is based on the variational convergence of an associated sequence of Dirichlet forms in the generalized Mosco sense of Kuwae and Shioya \cite{MR2015170}.
} ~\\



\section{Introduction} 
As shown in \cite{vrs07},   for  $\beta >0$ there exists a measure $\mathbb P^\beta$ and a Hunt process $\bigl(P_{\eta \in \P(\eI)},(\mu_t)_{t\geq 0}\bigr)$ 
on $(\P(\eI), \tau_w)$, the space  of Borel probabilities over $\eI$ equipped with the weak topology, such that 
\begin{itemize}
\item[i)]
$\mathbb P^\beta$ admits the formal representation  $\mathbb P^\beta (d\mu) = \frac 1 Z e^{-\beta \Ent(\mu) } \mathbb P^0(d\mu)$ as a Gibbs-type measure on $\P(\eI)$ with the Boltzmann entropy  $\Ent(\mu)= \int_\eI  \log ( d\mu/dx) d\mu$ as Hamiltonian and  
\item[ii)] $\bigl(P_{\eta \in \P(\eI)},(\mu_t)_{t\geq 0}\bigr)$  is a $\mathbb P^\beta $-symmetric diffusion on $(\P(\eI),\tau_w)$ with intrinsic distance given by the quadratic Wasserstein distance $d_2^W$. 
 \end{itemize}  

Moreover, letting denote by $(\mu_\cdot)$ the process obtained from the invariant starting distribution $\mathbb P^\beta$ we arrive at a solution of the following martingale problem. The  initial law of  $(\mu_t)_{t\geq 0}$ satisfies   
\begin{equation}
\langle f,\mu_0  \rangle \sim \int_0^1 f(D_t^\beta ) dt \quad \forall f\in C(\eI),
\label{initcond}
\end{equation}
where $(D_t^\beta)_{t\in \eI}$ is the real valued Dirichlet (or normalized Gamma) process over $\eI$ with parameter $\beta>0$, and for $f\in C^2(\eI)$ with $f'(0)=f'(1)=0$ the process
\begin{equation}\label{martpart}
\begin{split}
 M_t&=\langle f,\mu_t\rangle-\beta\cdot\int_0^t
\langle f'',\mu_s\rangle ds\\
&-\int_0^t\left(
\sum_{I\in\Gaps(\mu_s)}\left[\frac{f''(I_-)+f''(I_+)}2-\frac{f '(I_+)-f'(I_-)}{|I|}\right] -\frac{f''(0)+f''(1)}2\right)ds,
\end{split}
 \end{equation}
 is a continuous martingale 
with quadratic variation process
\begin{equation}
 [ M]_t=2\int_0^t\langle(f')^2,\mu_s\rangle ds. \label{qvar}
\end{equation}
Here  $\Gaps(\mu)$ denotes the  set of connected components in the complement of spt($\mu$). \pagebreak

Properties i) and ii)  suggest to view  $(\mu_t)_{t\geq 0}$ as model for a diffusing fluid when its heat flow is perturbed by a kinetically uniform random forcing. The actual construction of $(\mu_t)_{t\geq 0}$ in \cite{vrs07}  uses  abstract Dirichlet form methods \nocite{MR1303354} without direct reference to physical intuition. $\bigl(P_{\eta \in \P(\eI)},(\mu_t)_{t\geq 0}\bigr)$ is generated from the $L^2(\P(\eI), \mathbb P^\beta)$-closure $\E$ of the quadratic form  
\[  Q(F,F) = \int_{\P(\eI)} \norm{\nabla^w F}^2_\mu \mathbb P^\beta (d\mu), \quad F\in \mathcal Z
\]
 on the class $\mathcal Z =\{  F: \P(\eI) \to \R\,|\,F(\mu) = f(\langle \phi_1, \mu\rangle, \langle \phi_2, \mu\rangle,\dots, \langle \phi_k, \mu\rangle), f\in C^\infty_c(\R^k), \{\phi_i\}_{i=1}^k \subset  C_c^\infty (\R), k \in \N\}$, where   $\norm{\nabla^w F} _\mu = \norm{  (D_{|\mu} F)'(\cdot)}_{L^2(\eI, \mu)}$ and $(D_{|\mu} F)(x) =\partial_{t|t=0} F(\mu+t\delta_x)$. \smallskip 

In this paper we aim at an approximation of $(\mu_\cdot)$ by  a sequence of interacting particle systems in order to gain insight into some of its qualitative features.\smallskip

In analytic terms the Wasserstein diffusion $(\mu_\cdot)$ solves an SPDE with nonlinear (singular) drift and non-Lipschitz multiplicative noise. \nocite{MR1707314} It should be noted that the class of stochastic nonlinear evolution equations   admitting a rigorous particle approximation appears to be rather small. 
Some examples of lattice systems with stochastic nonlinear  hydrodynamic behaviour are reviewed in  \cite{MR1661764}, the case of exchangeable diffusions is studied e.g. in  
 \cite{MR964251,MR1337250} and  \cite{MR1404523,MR2164042}
deal with stochastic nonlinear scaling limits of  population models with interactive behaviour. 
\smallskip

Given  the singularity of the generator of $(\mu_\cdot)$, here  we choose an approximation by a sequence of  \textit{reversible} particle systems. This allows to use Dirichlet form methods for the passage to the limit instead of arguing along a sequence of martingale problems. For the identification of the limit we have to assume that $\E$ is a maximal element in the class of (not necessarily regular) Dirichlet forms on $L^2(\P(\eI), \mathbb P^\beta)$, i.e. that \textit{Markov uniqueness} holds for $\E$.  

\smallskip

The assumption on Markov uniqueness appears in several quite similar contexts as well \cite{MR2186217,MR2278454}. The verification is usually difficult, in particular in a non-Gaussian infinite dimensional setting involving singular logarithmic derivatives \cite{MR1734956}. Finally, by general principles the Markov uniqueness of $\E$ is weaker than the essential self-adjointness of the generator of $(\mu_t)_{t\geq 0}$ on $\mathcal Z$ and stronger than the well-posedness, i.e. uniqueness, of the martingale problem problem defined by (\ref{initcond}), (\ref{martpart}) and  (\ref{qvar}) in the class of Hunt processes on $\P(\eI)$, cf. \cite[theorem 3.4]{MR1335494}. 

\nocite{MR1658889} 

\section{Set Up and Main Result} 
For  $N\in \N$ let $X^N_t=(x^1_t, \cdots, x^{N-1}_t) \in  \Sigma_N := \{x \in \R^{N-1}, 0 \leq  x^1\leq  x^2\leq \cdots\leq  x^{N-1} \leq 1 \} \subset \R^{N-1}$ denote the ordered vector of the positions of  $N-1$ particles in $\eI$. 
Define the probability measure $q_N$ on $\Sigma_N$ by
\[ q_N (dx^1, \cdots, dx^{N-1}) = \frac {\Gamma(\beta)}{(\Gamma(\beta/N))^N} \prod_{i=1}^{N} (x^i-x^{i-1})^{\frac \beta N -1} dx^1\dots dx^{N-1},
\]
where $x_0=0$ and $x_N =1$ by convention. The $L^2(\Sigma_N, q_N)$-closure of 
\begin{equation*}
 \E^N(f,f) = \int_{\Sigma_N} | \nabla f|^2(x) q_N (dx) , \quad f \in C^\infty (\Sigma_N)
\label{findirichf}
\end{equation*}
defines a local regular Dirichlet form, which is again denoted by $\E^N$. Let $(X^N_t)_{t\geq 0}$ be the associated Markov process on $\Sigma_N$, starting from the invariant distribution $q_N$ and let
\[  \mu^N_t = \frac 1 {N-1} \sum_{i=1}^{N-1} \delta_{  x^{i}_{N \cdot t}} \in \P(\eI),\]
 be the associated empirical measure process on $\eI$, considered on time scale $N\cdot t$. Then we prove the following assertion.   
\begin{theorem} \label{mainthm} \textit{Assume Markov-uniqueness holds for $\E$, then  $(\mu^N_.)\stackrel{N \to \infty} {\Longrightarrow } (\mu.) $ in $C_{\R_+} \bigl((\P(\eI),\tau_w\bigr))$.}
\end{theorem}

\begin{remark} 
A careful integration by parts for $q_N$ shows that the domain the generator $L_N$ of $\E^N$ contains the set of all smooth Neumann functions on $\Sigma_N$. For such $f$  
\[ L^N f(x)= (\frac \beta N -1 )   \sum_{i=1}^{N-1}\left ( \frac 1 {x^i - x^{i-1}} - \frac 1 {x^{i+1} - x^{i}} \right) \frac {\partial }{\partial x^i} f(x) +  \Delta f(x)\quad \for x\in \mbox{Int}(\Sigma_N).\]   

Hence given initial conditions   $0< x^1_0< x_0^2 <\cdots < x^{N-1}_0< 1$, $(X^N_\cdot)$  is the formal  solution to  the  system of coupled Skorokhod SDEs
\begin{align}
d x^i _t = (\frac \beta N -1 ) \left ( \frac 1 {x^i_t - x^{i-1}_t} - \frac 1 {x^{i+1}_t - x^{i}_t} \right) dt +\sqrt 2 dw^i_t + dl_t^{i-1} -dl_t^i, \quad i=1, \cdots, N-1, \label{bessde1}
\end{align}
with independent real Brownian motions $\{w^i\}$ and local  times $l^i$  satisfying  
\begin{equation}
  dl_t^i \geq 0, \quad l_t^i = \int _0^t \one _{\{x^i_s=x_s^{i+1}\}} dl_s^i.
\label{bessde2}
\end{equation}
 $(X^N_\cdot)$ may thus be considered as system of coupled two sided real Bessel processes with uniform Bessel dimension $\delta= \frac \beta N$. Similar to the real Bessel process $BES(\delta)$ with Bessel dimension $\delta < 1$, the existence of $X^N$ is not a trivial fact. By analogy one should expect that the Skorokhod-SDE  defined by  (\ref{bessde1}) and (\ref{bessde2})  is ill-posed, but that nevertheless $\E^N$ generates a Feller semigroup on $\Sigma_N$.\\
\end{remark}
\begin{remark} For simulation the dynamics of $(X^N_\cdot)$  can be approximated by $X^{N,\epsilon}_t= \mathcal X^{N,\epsilon}_{\lfloor t/\epsilon^2\rfloor}$, $t\geq 0$, where $(\mathcal X^{N,\epsilon}_n)_{n\geq 0}$ is the Markov chain on $\Sigma_N$ with transition kernel $\mu^{N,\epsilon}(x,A) = \frac { q_N(B_\epsilon(x)\cap \Sigma_N\cap A)}  { q_N(B_\epsilon(x)\cap \Sigma_N)}$. An alternative approach via  a regularized version of the formal SDE  (\ref{bessde1}) and (\ref{bessde2}) was pursued by Theresa Heeg (Bonn). For illustration we present her results for the case of $N=4$ particles, starting from an equidistant configuration, with  $\beta=10, \beta =1$ and $\beta=0.3$ respectively, at large times. \renewcommand{\thesubfigure}{}
\begin{figure}[htbp]
	\begin{center}
\subfigure[$\beta=10$]{ \scalebox{0.35}{   \includegraphics{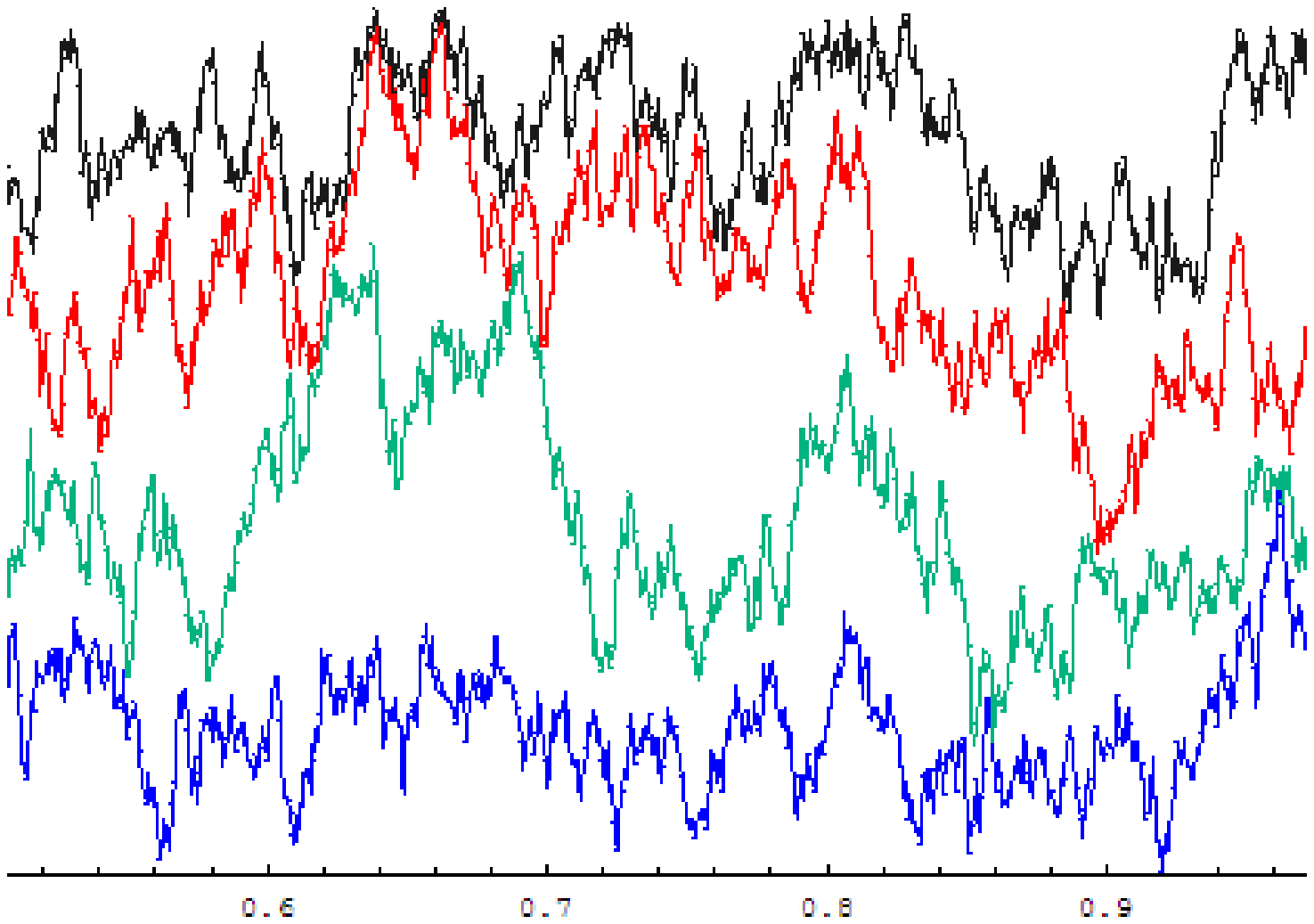} } }
\subfigure[$\beta=1$]{ \scalebox{0.35}{   \includegraphics{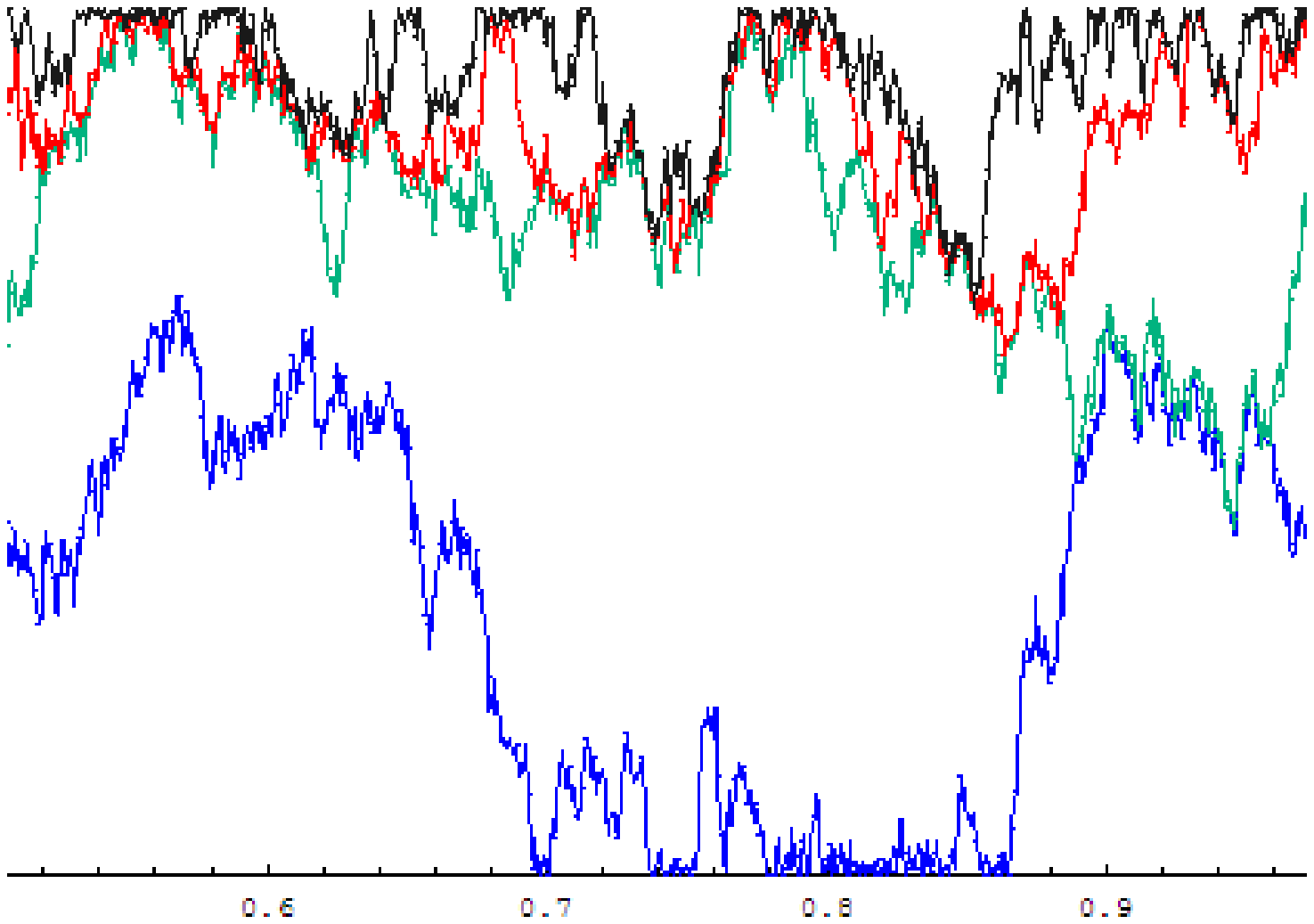} } }
\subfigure[$\beta=0.3$]{ \scalebox{0.35}{   \includegraphics{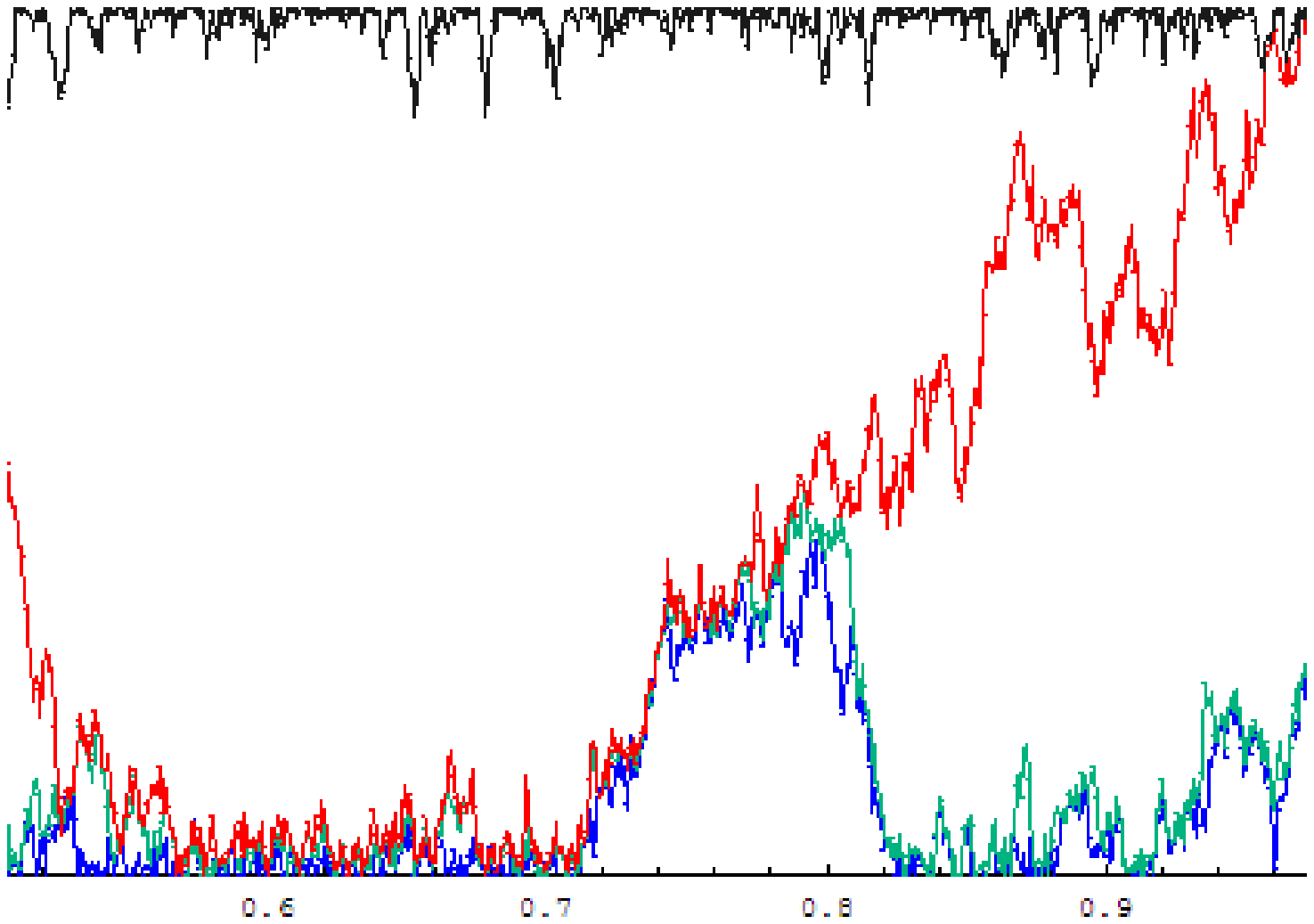} } }	
\end{center}
\label{fig:numsim}
\end{figure}
\pagebreak

\end{remark}

\section{Proof of theorem \ref{mainthm}} 
\subsection{Tightness}
As usual we show compactness of the laws of $(\mu_.^N)$ and, in a second step the uniqueness of the limit.
\begin{proposition} \label{tighness} \textit{The sequence $(\mu_.^N)$ is tight in $C_{\R_+}((\P(\eI),\tau_w))$.}
\end{proposition}

\begin{proof}
 According to theorem 3.7.1 in \cite{MR1242575} it is sufficient to show that the sequence $(\langle f, \mu_.^N\rangle)_{N \in \N}$ is tight, where   $f$  is taken from a dense subset  in $\mathcal F \subset C(\eI)$. Choose $\mathcal F : = \{ f \in C^3(\eI)\, |, f'(0)=f'(1)=0\}$, then 
$\langle f , \mu^N_t \rangle = F^N(X^N_{N\cdot t})$ 
with \[ F^N(x)=\frac 1 {N-1} \sum_{i=1}^{N-1} f(x^i).\]
The condition $f'(0)=f'(1)=0$ implies $F^N \in \D(L^N)$.  Moreover, for $\for x\in \mbox{Int}(\Sigma_N)$
\begin{align*}
N \cdot&  L^N F^N(x)  =  \frac \beta {N-1}   \sum_{i=1}^{N} \frac{f'(x^i)-f'(x^{i-1})}{x^i-x^{i-1}} \\
& + \frac N {N-1} \sum_{i=1}^{N-1} \left( f''(x^i)- \frac{f'(x^i)-f'(x^{i-1})}{x^{i}-x^{i-1}} \right) + \frac N {N-1}  \frac{f'(x^N)-f'(x^{N-1})}{x^N -x^{N-1}},
\end{align*}
such that 
\[  |N\cdot  L^N F^N(x)| \leq   \norm{f''}_\infty \frac {N \cdot (\beta  +1)} {N-1} +  \norm{f'''}_\infty \frac N {N-1}\leq C(\beta, \norm{f}_{C^3(\eI)}). \]
This implies a uniform in $N$ Lipschitz bound for the BV part in the Doob-Meyer decomposition of $F^N(X^N_{N.})$.  The process  $X^N$ has continuous sample paths with square field operator $\Gamma(F,F) = L (F^2) - 2F \cdot LF = |\nabla F|^2$. Hence
the quadratic variation of the martingale part of $F^N(X^N_{N\cdot})$ satisfies
\[ [F^N(X^N_{N\cdot})]_t- [F^N(X^N_{N\cdot})]_s =  N \cdot \int_s^t  |\nabla F^N|^2(X^N_s) ds = \frac N {(N-1)^2} \cdot \int_s^t  \sum_{i=1}^{N-1} (f')^2(x^i_s) ds \leq 2 (t-s)  \norm{f'}^2_\infty.  \] Since
\[
F^N(X^N_0)=\frac 1 {N-1} \sum_{i=1}^{N-1} f(D^\beta_{i/N}) \rightarrow \int_0^1 f(D_s^\beta) \, ds \qquad \text{$\Qbeta$-a.s.},
\]
 the law of $F^N(X^N_0)$ is convergent. Using now Aldous' tightness criterion in an appropriate version on sequences of semi-martingales the assertion follows, cf. corollary 3.6.7. in \cite{MR1242575}.\bbox
\end{proof}
\begin{remark}
Using the symmetry of $(X^N_\cdot)$  we could have used  the Lyons-Zheng decomposition for the tightness proof instead. The argument above shows the balance of  first and second order parts of $N \cdot L^N$ as $N$ tends to infinity.
\end{remark}

\subsection{Identification of the Limit}

\subsubsection{The $\G$-Parameterization} 
In order to identify the limit of the sequence $(\mu^N_.)$ we parameterize the space $\P(\eI)$  in terms of right continuous quantile functions, cf. \cite{vrs07}.  The set  \[  \G = \{ g: [0,1)\to \eI\, |\, g \mbox{ cadlag nondecreasing}\}, \] 
 equipped with the $L^2(\eI, dx)$ distance $d_{L^2}$ is a compact subspace of $L^2(\eI,dx)$.  It is homeomorphic to $(\P(\eI), \tau_w)$ by means of the map
\[ \rho:  \G \to \P(\eI),  \quad \quad g\to g_*(dx), \] which takes a function $g\in \G$ to the image measure of $dx$ under $g$. The inverse map $\kappa= \rho^{-1} : \P(\eI) \to \G$ is realized by taking the right continuous quantile function. 
\\

For technical reasons we introduce the following modification of $(\mu^N_.)$ which is better behaved in terms of the map $\kappa$.
\begin{lemma} \label{equivalence} \textit{ For $N \in \N$ define the Markov process   
\[ \nu ^N_t := \frac {N-1}{N} \mu^N_t + \frac 1 N \delta_0 \in \P(\eI),\] 
then  $(\nu^{N'}_.)$ is convergent  on $C_{\R_+}((\P(\eI),\tau_w))$ along any subsequence $N'$ if and only if $(\mu^{N'}_.)$ is. In this case both limits coincide.}
\end{lemma}
\begin{proof} For any $f\in C(\eI)$ the sequence $(\langle f, \mu^{N'}\rangle)_{N'}$ is tight if and only if the same holds true for the sequence  $(\langle f, \nu^{N'}\rangle)_{N'}$, where the limits coincide. Using Theorem 3.7.1 in \cite{MR1242575} again, this implies  $(\mu^{N'})$ is tight in case  $(\nu^{N'})$ is and vice versa. Since the map $l_f: C_{\R_+}(\P(\eI)) \to C_{\R_+}(\R)$, $(m_t)_{t\geq 0} \to (\langle m_t,f\rangle)_{t\geq 0}$ is continous for $f \in C(\eI)$ we conclude that the respective laws of $l_f$ on $C_{\R\geq 0}(\R)$ induced by any two potential limits of $(\mu^{N'}_.)$ and $(\nu^{N'}_.)$ coincide. Hence those limits must in fact be identical. \bbox
\end{proof}

Let $(g^N_\cdot) := (\kappa (\nu^N_\cdot))$ be the process $(\nu^N_\cdot)$ in the $\G$-parameterization. It can also be  obtained by 
\[g^N_t  = \iota (X^N_{N\cdot t}) \]
with the imbedding $\iota = \iota ^N$
\[\iota :\Sigma_N \to \G,\quad \quad \iota (x)= 
\sum_{i=0}^{N-1} x^i \cdot \one_{[i/N,(i+1)/N)}.
\] 

Similarly, let   $(g_\cdot)=(\kappa(\mu_.))$ be the $\G$-image of the Wasserstein diffusion under the map $\kappa$ with  invariant initial distribution $\Q^\beta$. In \cite[theorem 7.5]{vrs07} it is shown that  
 $(g_\cdot)$ is generated by the Dirichlet form, again denoted by  $\E$, which is obtained as the 
$L^2(\G, \Qbeta)$-closure of  
\[\label{Wasserstein-dir-int}
\E (u,v)=\int_{\G} \langle\nabla u_{|g}(\cdot),  \nabla v_{|g}(\cdot)\rangle_{L^2(\eI)} \,  \Qbeta(dg), \quad \quad   u,v\in \Cyl^1(\G).
\]
 on  the class 
\[\Cyl^1(\G)= \{u : \G \to \R\, |\, u(g)=U ( \langle
f_1, g\rangle_{L^2 }, \dots ,\langle
f_m, g\rangle_{L^2} ),   U \in C^1_c(\R^m), \{f_i\}_{i=1}^m \subset L^2(\eI),    m\in \N\},\] 
where  $\nabla u_{|g} $  is  the $L^2(\eI, dx)$-gradient of $u$ at $g$.
\smallskip

The convergence of  $(\mu^N_\cdot)$ to $(\mu_\cdot)$  in $C_{\R_+}(\P(\eI),\tau_w)$ is thus equivalent to the convergence of  $(g^N_\cdot)$ to $(g_\cdot)$  in   $C_{\R_+}(\G,d_{L^2})$. By proposition \ref{tighness} and lemma \ref{equivalence}  $(g^N_\cdot)_N$ is a tight sequence of processes on $\G$. The following statement idenitifies  $(g_\cdot)$ as the unique weak limit. 
\begin{proposition} \label{identlimit}
\textit{Let $\E$ be Markov-unique. Then for any  $f\in C(\G^l)$ and $0\leq t_1< \ldots < t_l$, $\mathbb E(f(g^N_{t_1}, \cdots,g^N_{t_l}))  \stackrel{N\to \infty}{\longrightarrow } \mathbb E(f(g_{t_1}, \cdots,g_{t_l}))$.} 
\end{proposition}


\subsubsection{Finite Dimensional Approximation of Dirichlet Forms in Mosco Sense}

Proposition \ref{identlimit} is proved by showing that the sequence of generating Dirichlet forms  $N\cdot \E^N$ of $(g^N_\cdot)$ on  $L^2(\Sigma_N, q_N)$ converges to  $\E$ on $L^2(\G,\Q)$ in the generalized Mosco sense of Kuwae and Shioya, allowing for varying base $L^2$-spaces. We recall the framework developed in \cite{MR2015170}.

\begin{definition}[Convergence of Hilbert spaces]
\textit{
A sequence of Hilbert spaces $H^N$ converges to a Hilbert space $H$ if there exists a family of linear maps $\{\Phi^N: \, H\rightarrow H^N\}_N$ such that
\[ 
\lim_N \norm{\Phi^N u}_{H^N}=\norm{u}_H, \qquad \text{for all $u\in H$}.
\] 
A sequence $(u_N)_N$ with $u_N \in H_N$  converges strongly to a vector $u\in H$ if there exists a sequence $(\tilde u_N)_N\subset H$ tending to $u$ in $H$ such that
\[
\lim_N \limsup_M \norm{\Phi^M \tilde u_N-u_M}_{H^M}=0,
\]
and $(u_N)$  converges weakly to $u$ if
\[
\lim_N \langle u_N,v_N\rangle_{H^N}=\langle u,v \rangle_H,
\]
for any sequence $(v_N)_N$ with $v_N\in H^N$ tending strongly to $v\in H$. Moreover, a sequence $(B_N)_N$ of bounded operators on $H^N$ converges strongly (resp.\ weakly) to an operator $B$ on $H$ if $B_N u_N \rightarrow Bu$ strongly (resp.\ weakly) for any sequence $(u_N)$ tending to $u$ strongly (resp.\ weakly).} 
\end{definition}

\begin{definition}[Mosco Convergence]
\textit{
A sequence $(E^N)_N$ of quadratic forms $E^N$ on $H^N$ converges to a quadratic form $E$ on $H$ in the Mosco sense if the following two conditions hold:
\begin{description}
\item[Mosco I:] If a sequence $(u_N)_N$ with $u_N\in H^N$ weakly converges to a $u\in H$, then
\[
E(u,u)\leq \liminf_N E^N(u_N,u_N). 
\]
\item[Mosco II:] For any $u\in H$ there exists a sequence $(u_N)_N$ with $u_N\in H^N$ which converges strongly to $u$ such that
\[
E(u,u)=\lim_N E^N(u_N,u_N).
\]
\end{description}}
\end{definition}

Extending  \cite{MR1283033} it is shown in \cite{MR2015170} that Mosco convergence of a sequence of Dirichlet forms is equivalent to the strong convergence of the  associated resolvents and semigroups. We will apply this result when  $H^N= L^2(\Sigma_N, q_N)$, $H= L^2(\G, \Qbeta)$  and  $\Phi^N$ is defined to be the conditional expectation operator   
\[ \Phi^N: H \to H^N; \quad \quad   (\Phi^N u )(x) := \mathbb E(u | g_{i/N} = x_i, i=1, \dots, N-1).\] 
However, we shall prove that the sequence $N\cdot \E^N$ converges to $\E$ in the Mosco sense in a slightly modified fashion, namely the condition (Mosco II) will be replaced by
\begin{description}
\item[Mosco II':] \textit{There is a core $K\subset \D(E)$ such that for any $u\in K$ there exists a sequence $(u_N)_N$ with $u_N\in \D(E^N)$ which converges strongly to $u$ such that $E(u,u)=\lim_N E^N(u_N,u_N)$.}
\end{description}

\begin{theorem} \textit{Under the assumption that $H^N \to H$ the conditions (Mosco I) and (Mosco II') are equivalent to the strong convergence of the associated resolvents.} 
\end{theorem}
\begin{proof}
We proceed as in the proof of theorem 2.4.1 in \cite{MR1283033}. By theorem 2.4 of \cite{MR2015170}  strong convergence of resolvents implies Mosco-convergence in the original stronger sense. Hence  we need to show only that  our weakened notion of Mosco-convergence also implies strong convergence of resolvents.

Let $\{R^N_\lambda, \, \lambda>0\}$ and $\{R_\lambda, \, \lambda>0\}$ be the resolvent operators associated with $E^N$ and $E$, respectively. Then, for each $\lambda >0$ we have to prove that for every $z\in H$ and every sequence $(z_N)$ tending strongly to $z$ the sequence $(u_N)$ defined by $u_N:=R^N_\lambda z_N \in H^N$ converges strongly to $u:=R_\lambda z$ as $N\rightarrow \infty$. The vector $u$ is characterized as the unique minimizer of $E(v,v)+\lambda \langle v,v \rangle_H -2 \langle z, v \rangle_H$ over $H$ and a similar characterization holds for each $u_N$. Since for each $N$ the norm of $R^N_\lambda$ as an operator on $H^N$ is bounded by $\lambda^{-1}$, by Lemma 2.2 in \cite{MR2015170} there exists a subsequence of $(u_N)$, still denoted by $(u_N)$, that converges weakly to some $\tilde u \in H$. By (Mosco II') we find for every $v\in K$ a sequence $(v_N)$ tending strongly to $v$ such that $\lim_N E^N(v_N,v_N)=E(v,v)$. Since for every $N$
\[
E^N(u_N,u_N)+\lambda \langle u_N, u_N \rangle_{H^N} -2 \langle z_N, u_N \rangle_{H^N}
\leq E^N(v_N,v_N)+\lambda \langle v_N, v_N \rangle_{H^N} -2 \langle z_N, v_N \rangle_{H^N},
\]
using the condition (Mosco I) we obtain in the limit $N\rightarrow \infty$:
\[
E(\tilde u,\tilde u)+\lambda \langle \tilde u, \tilde u \rangle_{H} -2 \langle z, \tilde u \rangle_{H}
\leq E(v,v)+\lambda \langle v, v \rangle_{H} -2 \langle z, v \rangle_{H},
\]
which by the definition of the resolvent together with the density of $K  \subset  D(E)$ implies that $\tilde u = R_\lambda z=u.$ This establishes the weak convergence of resolvents. It remains to show strong convergence. Let $u_N = R^N_\lambda z_N$ converge weakly to $u = R_\lambda z$ and choose $ v \in K$ with the respective strong approximations $ v_N \in H^N$ such that $E^N(v_N,v_N)\to E( v,  v)$, then the resolvent inequality for $R^N$ yields 
\[ E^N(u_N,u_N) +\lambda \norm{u_N-z_N/\lambda}^2_{H^N} \leq E^N (v_N,v_N) +\lambda \norm{v_N-z_N/\lambda}^2_{H^N}.\]
Taking the limit for $N \to \infty$, one obtains 
\[ \limsup_N \lambda \norm{u_N-z_N/\lambda}^2_{H^N} \leq E (v,v)- E(u,u)  +\lambda \norm{v-z/\lambda}^2_H.\]
Since $K$ is a dense subset we may now let $v \to u \in D(E)$, which yields 
\[ \limsup_N  \norm{u_N-z_N/\lambda}^2_{H^N} \leq  \norm{u-z/\lambda}^2_H.\] Due to the weak lower semicontinuity of the norm  this yields $\lim_N \norm{u_N-z_N/\lambda}=  \norm{u-z/\lambda}$. Since strong  convergence in $H$ is equivalent to weak convergence together with the  convergence of the associated norms the claim follows (cf.\ Lemma 2.3 in \cite{MR2015170}). \bbox
\end{proof}

Proposition \ref{identlimit} will now essentially be implied by the following statement, which by the definitions above summarizes the subsequent three propositions.

\begin{theorem} \label{mconvergence}
\textit{Assume that  $\E$ is Markov-unique on  ${L^2(\G,\Q)}$. Then  $(N \cdot \E^N, H^N)$ converges to $(\E, H)$ along $\Phi^N$ in Mosco sense.} 
\end{theorem}

\begin{proposition} \label{hconvergence} \textit{$H^N $ converges to $H$ along $\Phi^N$, for $N\to \infty$.}
\end{proposition}

\begin{proof}
We have to show that $\norm{\Phi^N u}_{H^N} \to \norm{u}_{H}$  for each $u \in H$. Let  $\mathcal F^N$ be the $\sigma$-Algebra on $\G$ generated by the projection maps $\{g \to g({i/N})\,|\, i =1, \dots, N-1\}$. By abuse of notation we identify  $\Phi^N u \in H$  with $\mathbb E(u|\mathcal F^N)$ of $u$, considered as an element of $L^2(\Qbeta, \mathcal F^N) \subset H$. Since the measure $q_N$ coincides with the respective finite dimensional distributions of $\Qbeta$ on $\Sigma_N$ we have $\norm {\Phi^N u}_{H^N} = \norm {\Phi^N u}_{H}$. Hence the claim will follow once we show that $\Phi^N u \to u$ in $H$. For the  latter we use the following abstract result,  whose proof  can be found, e.g. in  \cite[lemma 1.3]{MR1619139}. 
\begin{lemma}\label{condconv}
\textit{Let $(\Omega, \mathcal D , \mu )$ be a measure space and $(\mathcal F_n)_{n \in \N}$ a sequence of $\sigma$-subalgebras of $\mathcal D$. Then $E(f|\mathcal F_n) \to f$ for all $f \in L^p$, $p \in [1,\infty)$ if and only if for all $A \in \mathcal D $ there is a sequence $A_n \in \mathcal F_n $ such that $\mu(A_n \Delta A) \to 0$ for $n \to \infty$.}
\end{lemma}
In order to apply this lemma to the given case $(\G,  \mathcal B (\G),\Qbeta)$, where $\mathcal B (\G)$ denotes the  Borel $\sigma$-algebra on $\G$, let $\mathcal F_\Qbeta \subset \mathcal B(\G)$ denote the collection of all Borel sets $F\subset \G$ which can be approximated by elements $F_N \in \mathcal F^N$ with respect to $\Qbeta$ in the sense above. Note that $\mathcal F_\Qbeta $ is again a $\sigma$-algebra, cf. the appendix in \cite{MR1619139}. Let $\mathcal M $ denote the system of finitely based  open cylinder sets in $\G$ of the form $M = \{ g \in \G| g_{t_i} \in O_i, i= 1, \dots, L \} $ where $t_i \in \eI $ and $O_i \subset \eI $ open.  From  the almost sure right continuity of $g$ and the fact that  $g_.$ is continuous at $t_1, \dots, t_L$ for $\Qbeta$-almost all $g$ it follows that  
$M_N := \{ g \in \G| g_{(\lceil  t_i\cdot N \rceil /N )  } \in O_i, i= 1, \dots, L \} \in \mathcal F^N$ is an approximation of $M$ in the sense above.  Since $\mathcal M$ generates $\mathcal B(\G)$ we obtain $\mathcal B (\G) \subset \mathcal F_\Qbeta$ such that the assertion holds, due to  lemma \ref{condconv}.\bbox
\end{proof}

\begin{remark} It is much simpler to prove proposition \ref{hconvergence} for  a dyadic subsequence ${N'}=2^m$, $m \in \N$ when the sequence  $\norm {\Phi^{N'} u}_{H^{N'}}$ is nondecreasing and bounded, because $\Phi^{N'}$ is a projection operator in $H$ with increasing range im$(\Phi^{{N'}})$ as   ${N'}$ grows. Hence, $\norm {\Phi^{N'} u}_{H^{N'}}$ is Cauchy and thus \[ \norm {\Phi^{N'} u -\Phi^{M'} u}_H^2 =   \norm {\Phi^{N'} u}_H^2 -\norm{\Phi^{M'} u}_H^2 \to 0 \quad \mbox{for}\, M', {N'}  \to \infty,\]
i.e.\ the sequence $\Phi^{N'} u$ converges to some $v \in H$. Since obviously $\Phi^N u \to u$ weakly in $H$ it follows that $u=v$ such that the claim is obtained from  $|\norm {\Phi^{N'} u }_H- \norm {u }_H| \leq\norm {\Phi^{N'} u -u }_H$.
\end{remark}

To simplify notation for $f \in L^2(\eI,dx)$ denote the functional  $  g \to \langle f, g\rangle_{L^2(\eI)}$ on $\G$ by $l_f$. We introduce the set $K$ of polynomials defined by
\[
K=\left\{ u \in C(\G) \, | \, u(g)= \prod_{i=1}^n l_{f_i}^{k_i}(g), \, k_i \in \N,\, f_i \in C([0,1]) \right\}.
\]

\begin{corollary} \label{simplifycorr}
\textit{  For  a polynomial $ u \in K$ with  $u(g)= \prod_{i=1}^n l_{f_i}^{k_i}(g)$   let $u_N:= \prod_{i=1}^n  \bigl(\Phi^N( l_{f_i})\bigr)  ^{k_i} \in H^N$,
then $u_N \to u$ strongly.} 
\end{corollary}

\begin{proof}
Let $\tilde u^N  :=\prod_{i=1}^n  \bigl(\Phi^N( l_{f_i})\bigr)  ^{k_i} \in H$ be the respective product of conditional expectations, where as above $\Phi^N$ also denotes the projection operator  on $H=L^2(\G,\Qbeta)$. Since each of the factors  $\Phi^N( l_{f_i})\in H $   is uniformly bounded and converges strongly to   $l_{f_i}$ in $L^2(\G,\Qbeta)$, the convergence also holds true in any $L^p(\G,\Qbeta)$ with $p >0$. This implies  $\tilde u^N \to u$ in $H$. Furthermore,
 \begin{align}
\lim_N \lim_M  \norm {\Phi^M \tilde u_N -u_M}_{H^M} & = \lim_N \lim_M  \norm {\Phi^M \left(\prod_{i=1}^n  \bigl(\Phi^N( l_{f_i})\bigr)  ^{k_i}\right) -\prod_{i=1}^n  \bigl(\Phi^M( l_{f_i})\bigr)  ^{k_i}}_{H}  \nonumber \\
 & = \lim_N  \norm { \prod_{i=1}^n  \bigl(\Phi^N( l_{f_i})\bigr)  ^{k_i}  -\prod_{i=1}^n    l ^{k_i}_{f_i}  }_{H} =0. \tag*{$\Box$}
 \end{align}
\end{proof}

\begin{proposition}[Mosco II'] \textit{There is a core $K \subset \DE$ such that for all $u \in K$ there is a sequence $u_N \in \D(\E^N)$ converging strongly to $u \in H$ and   $N\cdot \E^N(u_N, u_N) \to 
\E(u,u)$. }
\end{proposition}

\begin{proof} It follows from the chain rule for the $L^2$-gradient operator $\nabla$ that the linear span of polynomials of the form  $u(g)= \prod_{i=1}^n l_{f_i}^{k_i}(g)$  with $ k_i \in \N$, $f_i \in C([0,1])$, $k_i \in \N$, is a core of $\DE$. Hence it suffices to prove the claim for such $u$. Let $u_N:= \prod_{i=1}^n  \bigl(\Phi^N( l_{f_i})\bigr)  ^{k_i} \in H^N$ as above then the strong convergence of $u^N $ to $u$ is assured by corollary \ref{simplifycorr}. 
From lemma \ref{condexp} below we obtain that  $\Phi^N( l_{f})(X)= \langle f,g_X\rangle$. In particular
 \[\bigl(\nabla \Phi^N( l_{f})(X)\bigr)^i= \frac 1 N \cdot \bigl(\eta^N*f\bigr)(\frac i N),\] 
where $\eta^N$ denotes the convolution kernel  $t \to \eta^N(t) =  N \cdot (1-\min(1,| N\cdot t|))$.
By this the convergence of $N \cdot \E_N (u^N, u^N) $ to $\E(u,u)$   follows easily from Lebesgue's dominated convergence theorem in $L^2(\G\times \eI, \Qbeta\otimes dx)$.  \bbox
\end{proof}

\begin{remark}\label{rem_Mosco2}
For later use we observe that for $u$ and $u_N$ as above and for $\Qbeta$-a.e.\ $g$ we have
\[
\norm{N \iota^N(\nabla u_N(g(1/N),\ldots, g((N-1)/N)))-\nabla u_{|g}}_{L^2(0,1)}\rightarrow 0 \quad \text{as $N\rightarrow \infty$},
\]
with $\iota^N: \mathbb{R}^{N-1} \rightarrow D([0,1),\mathbb{R})$ defined as above.
\end{remark}

\begin{lemma}\label{condexp} \textit{
For $X \in \Sigma_N$ define  $g_{X} \in \G$ by 
\[g_X(t) = x_i +(N\cdot t-i)(x_{i+1}-x_i) \quad \mbox{ if } t \in [\frac i N, \frac{i+1}N), \quad i=0, \dots, N-1,\] 
then  
\[
\mathbb E(g|\mathcal F_N)(X)=g_{X}.\]
}
\end{lemma}
\begin{proof}
The statement is a simple consequence of the explicit formula for the finite dimensional distributions of the Dirichlet process, cf. [vRS07]. \bbox
\end{proof}

For the verification of Mosco I we exploit that the respective integration by parts formulas of $\E^N$ and $\E$ converge. In case of a fixed state space a similar approach is discussed in \cite{MR2186217}. 
\smallskip

Let   $T^N:= \{ f: \Sigma_N \to \R^{N-1}\}$ be equipped with the norm 
\[ \norm {f}_{T^N}^2 : = \frac{1}{N} \int _{\Sigma_N} \norm{f(x)}_{R^{N-1}}^2 q_N (dx),  \]
then the corresponding integration by parts formula for $q_N$ on $\Sigma_N$  reads 
\begin{equation}
 \langle \nabla u , \xi\rangle_{T^{N}} =  - \frac 1 N \langle u,\div_{q_N} \xi\rangle_{H^N}. \label{enddivibp}
\end{equation}

To state the corresponding  formula for $\E$ we introduce the Hilbert space of vector fields on $\G$ by
\[T =L^2(\G\times \eI, \Qbeta\otimes dx),\] 
with dense subset $\Theta \subset T$
\[\Theta = \mbox{span}\{ \zeta  \in T\,|\, \zeta(g,t)=w(g)\cdot \varphi(g(t)),   w\in K,  \varphi\in\C^\infty(\eI):   \varphi(0)=\varphi(1)=0 \}.\]
 The $L^2$-derivative operator $\nabla$ defines a map 
\[ \nabla: \Cyl^1(\G) \to T \] 
which by \cite[proposition 7.3]{vrs07}, cf. \cite{ryz}, satisfies the following integration by parts formula, .
\begin{equation}
 \langle \nabla u , \zeta \rangle _T = -\langle u, \div_\Qbeta{\zeta}\rangle_H, \quad u \in \Cyl^1(\G), \zeta \in \Theta ,
\label{divibp} 
\end{equation}
where, for $\zeta (g,t) = w(g)\cdot \varphi(g(t))$,
\[ 
\div_\Qbeta \zeta (g) =  w(g)\cdot  V^\beta_\varphi(g) +    \langle \nabla w(g)(.) ,\varphi(g(.))\rangle_{L^2(dx)}\]
with 
$$V^\beta_\varphi(g):=V^0_\varphi(g)+\beta\int_0^1\varphi'(g(x))dx-\frac{\varphi'(0)+\varphi'(1)}{2}$$
and 
\[\label{drift-0}
V^0_\varphi(g):=\sum_{a\in J_g}\left[
\frac{\varphi'(g(a+))+\varphi'(g(a-))}2-\frac{\delta(\varphi\circ
g)}{\delta g}(a)\right].
\]
Here $J_g \subset [0,1]$ denotes the set of jump locations of $g$ and
\begin{equation*}\label{delta-def}
\frac{\delta (\varphi\circ g)}{\delta
g}\left(a\right):=\frac{\varphi\left(g(a+)\right) -
\varphi\left(g(a-)\right)}{g(a+) - g(a-)} \ .
\end{equation*}
By formula (\ref{divibp}) one can extend $\nabla$  to a closed operator on $D(\E)$ such that $\E(u,u)= \norm{\nabla u}_T^2$. The Markov uniqueness of $\E$ now implies the converse which is a characterization of $\DE$ via (\ref{divibp}).
\begin{lemma}[Meyers-Serrin property]
\label{myersserrin}\textit{
Assume Markov-uniqueness holds for $\E$, then 
\begin{equation} (\E(u,u))^{1/2} = \sup_{\zeta \in \Theta }
 \frac{\langle u, \div_\Qbeta{\zeta}\rangle_H}{\norm{\zeta}_T}. \label{mseq}\end{equation}}
\end{lemma}
\begin{proof}
We repeat the standard argument, cf. \cite{MR1734956}. Denoting the r.h.s. of (\ref{mseq}) by $(\hat \E(u,u))^{1/2}$ one obtains that $\hat \E$ is a Markovian extension of $\E$. Since  $\E$ is assumed maximal in the class of Markovian forms it follows $\E = \hat \E$.  \bbox
\end{proof}

The convergence of (\ref{enddivibp}) to (\ref{divibp}) is established by the following lemma  whose prove is given below.
 
\begin{lemma} \label{divapprox} \textit{For  $\zeta\in \Theta$  there exists a sequence of vector fields 
 $\zeta_N : \Sigma_N \to \R^{N-1}$ such that $\div_{q_N}\zeta^N \in H^N $ converges strongly to $\div_\Qbeta{\zeta} $ in $H$ and such that $\norm {\zeta^N}_{T^N} \to \norm{\zeta}_T$ for $N \to \infty$.}
\end{lemma}

\begin{proposition}[Mosco I]\label{mosco1} \textit{Let $\E$ be Markov-unique and let $u_N \in \D(\E^N)$ converge weakly to   $u \in H$, then 
\[\E(u,u) \leq  \liminf_{N \to \infty} N\cdot \E^N(u_N, u_N).\] }
\end{proposition}

\begin{proof}  
Let $u \in H$ and $u_N \in H^N$ converge weakly to $u$.
Let $\zeta \in \Theta$ and $\zeta^N$ be as in lemma $\ref{divapprox}$, then 
\begin{align*}
 \frac{-\langle u, \div_\Qbeta{\zeta}\rangle_H}{\norm{\zeta}_T} & = \lim \frac{-\langle u_N, \div_{q_N}{\zeta_N}\rangle_{H^N}}{\norm{\zeta^N}_{T^N}}\\
& =  \lim N \cdot \frac{\langle \nabla u_N,  {\zeta_N}\rangle_{T^N}}{\norm{\zeta^N}_{T^N}}
\leq \liminf N \cdot \norm{\nabla u_N}_{T^N}
= \liminf \left( N\cdot \E^N(u_N,u_N)\right)^{1/2}, 
\end{align*}
such that, using (\ref{mseq}), \begin{equation}
\left(\E(u,u)\right)^{1/2} = \sup_{\zeta \in \Theta} \frac{-\langle u, \div_\Qbeta{\zeta}\rangle_H}{\norm{\zeta}_T} \leq \liminf \left( N\cdot \E^N(u_N,u_N)\right)^{1/2}. \tag*{$\Box$}
\end{equation}
  
\end{proof}

\noindent
\textit{Proof of lemma \ref{divapprox}.} By linearity it suffices to consider  the case  $\zeta(g,t)=w(g) \cdot \varphi(g(t))$ with $w(g)=\prod_{i=1}^n l_{f_i}^{k_i}(g)$. Choose 
\[
(\zeta^N(x_1, \cdots, x_{N-1}))^{i} := w_N(x_1, \dots, x_{N-1})\cdot  \varphi(x_i)\] 
with $w_N:=\prod_{i=1}^n (\Phi^N(l_{f_i}))^{k_i}$.
Then
\[\div_{q_N} \zeta^N=w_N \cdot V_{N,\varphi}^\beta + \langle \nabla w_N, \vec\varphi \rangle_{\mathbb{R}^{N-1}},\]
 with
\[\vec \varphi(x_1,\ldots,x_{N-1}):=(\varphi(x_1),\ldots,\varphi(x_{N-1}))\]
 and
\[V_{N,\varphi}^\beta(x_1,\ldots,x_{N-1}):= (\frac{\beta}{N}-1) \sum_{i=0}^{N-1} \frac{\varphi(x_{i+1})-\varphi(x_i)}{x_{i+1}-x_i}+\sum_{i=1}^{N-1}\varphi'(x_i).
\]
We recall that for all bounded measurable $u: [0,1]^{N-1}\rightarrow \mathbb{R}$
\begin{align*}
\int_{\Sigma_N} u(x_1,\ldots,x_{N-1})\, q_N(dx)=\int_{\G} u(g(t_1),\ldots, g(t_{N-1})) \, \mathbb{Q}^\beta(dg),
\end{align*}
with $t_i=i/N$, $i=0,\ldots, N$. Using this we get immediately
\begin{align*}
\norm {\zeta^N}_{T^N}^2&=\frac{1}{N} \int_{\Sigma_N} \sum_{i=1}^{N-1} w_N^2(x) \, \varphi(x_i)^2 \, q_N(dx) 
= \int_{\G}   w_N^2(g(t_1),\ldots, g(t_{N-1})) \frac{1}{N} \sum_{i=1}^{N-1} \, \varphi(g(t_i))^2 \, \mathbb{Q}^\beta(dg) \\
&\rightarrow   \int_{\G}  w^2(g)  \int_0^1 \varphi(g(s))^2 \,ds  \, \mathbb{Q}^\beta(dg)=\norm{\zeta}_{T}^2. 
\end{align*}
To prove strong convergence of $\div_{q_N}\zeta^N $  to $\div_\Qbeta{\zeta} $, by definition we have to show that there exists a sequence $(d^N\zeta)_N\subset H$ tending to $\div_\Qbeta{\zeta} $ in $H$ such that 
\begin{align*}
&\lim_{N} \limsup_{M} \norm{ \Phi^M (d^N\zeta )-\div_{q_M}\zeta^M}^2_{H^M}=0. 
\end{align*}
The choice
\begin{align*}
d^N\zeta(g):=\div_{q_N}\zeta^N(g(t_1),\ldots,g(t_{N-1}))
\end{align*}
makes this convergence trivial, once we have proven that in fact $(d^N\zeta)_N$ converges to $\div_\Qbeta{\zeta} $ in $H$.
This is carried out in the following two lemmas. \bbox

\begin{lemma}
For $\Qbeta$-a.s.\ $g$ we have
\begin{align*}
V_{N,\varphi}^\beta(g(t_1),\ldots, g(t_{N-1})) \rightarrow V_\varphi^\beta(g), \quad \text{as } N \rightarrow \infty,
\end{align*}
and we have also convergence in $L^p(\G,\Qbeta)$, $p>1$.
\end{lemma}

\begin{proof}
We rewrite $V_{N,\varphi}^\beta(g(t_1),\ldots, g(t_{N-1}))$ as
\begin{align} \label{V_n_rewritten}
V_{N,\varphi}^\beta(g(t_1),\ldots, g(t_{N-1}))=&
\beta \sum_{i=0}^{N-1} \frac{\varphi(g(t_{i+1}))-\varphi(g(t_i))}{g(t_{i+1})-g(t_i)} (t_{i+1}-t_i)  \nonumber \\
&  -\frac{\varphi(g(t_{1}))-\varphi(g(t_0))}{g(t_{1})-g(t_0)}+\sum_{i=1}^{N-2} \left( \varphi'(g(t_i)) -\frac{\varphi(g(t_{i+1}))-\varphi(g(t_i))}{g(t_{i+1})-g(t_i)} \right)\\
& +\varphi'(g(t_{N-1})) -\frac{\varphi(g(t_{N}))-\varphi(g(t_{N-1}))}{g(t_{N})-g(t_{N-1})} \nonumber .
\end{align}
Note that all terms are uniformly bounded in $g$ with a bound depending on the supremum norm of $\varphi'$ and $\varphi''$, respectively. Since the same holds for $V^\beta_\varphi(g)$ (cf.\ Lemma 5.1 in [vRS07]), it is sufficient to show convergence $\Qbeta$-a.s. By the support properties of $\Qbeta$ $g$ is continuous at $t_N=1$, so that  the last line in \eqref{V_n_rewritten} tends to zero. Using Taylor's formula we obtain that  the first term in \eqref{V_n_rewritten} is equal to
\begin{align*}
\beta \sum_{i=0}^{N-1} \varphi'(g(t_i)) (t_{i+1}-t_i)
+\frac{1}{2} \sum_{i=0}^{N-1}   \varphi''(\gamma_i) \, (g(t_{i+1})-g(t_i)) \, (t_{i+1}-t_i),
\end{align*}
for some $\gamma_i \in [g(t_i),g(t_{i+1})]$. Obviously, the first term tends to $\beta \int_0^1 \varphi'(g(s))\, ds$ and the second one to zero as $N\rightarrow \infty$. Thus, it remains to show that the second line in \eqref{V_n_rewritten} converges to 
\begin{align} \label{target_limit}
\sum_{a\in J_g}\left[
\frac{\varphi'(g(a+))+\varphi'(g(a-))}2-\frac{\delta(\varphi\circ
g)}{\delta g}(a)\right]-\frac{\varphi'(0)+\varphi'(1)}{2}.
\end{align}
Note that by the right-continuity of $g$ the first term in the second line in \eqref{V_n_rewritten} tends to $-\varphi'(0)$.  Let now $a_2, \dots, a_{l-1}$ denote the $l-2$ largest jumps of
$g $ on $ ]0,1[$. For $N$ very large (compared with $l$) we may assume that $a_2, \dots, a_{l-2}  \in ]\frac 2 N , 1 -\frac 2 N [$. Put $a_1 := \frac 1 N $, $a_{l} := 1- \frac 1 N$. 
For $j=1, \dots, l$ let $k_j$ denote the index $i \in \{ 1, \dots,
N-1\}$, for which $a_j \in \left[t_i, t_{i+1} \right[$. In particular,   $k_1 =1 $ and $k_l=N-1$. Then
\begin{align} \label{sum_k_i}
\sum_{i\in \{k_2,\ldots,k_{l-1}\}} \varphi'(g(t_i)) -\frac{\varphi(g(t_{i+1}))-\varphi(g(t_i))}{g(t_{i+1})-g(t_i)}&
\xrightarrow[N\rightarrow \infty]{} \sum_{j=2}^{l-1} \varphi'(g(a_j-))-\frac{\delta(\varphi\circ
g)}{\delta g}(a_j) \nonumber \\
&\xrightarrow[l\rightarrow \infty]{} \sum_{a\in J_g} \varphi'(g(a-))-\frac{\delta(\varphi\circ
g)}{\delta g}(a).
\end{align}
Provided $l$ and $N$ are chosen so large that
\begin{align*}
|g(t_{i+1})-g(t_i)|\leq \frac {C} {l}
\end{align*}
for all $i\in \{0,\ldots, N-1\}\backslash \{k_1,\ldots,k_l\}$, where $C=\sup_s |\varphi'''(s)| /6$,
again by Taylor's formula we get for every $j\in\{1,\ldots,l-1\}$
\begin{align*}
&\sum_{i=k_j+1}^{k_{j+1}-1} \varphi'(g(t_i)) -\frac{\varphi(g(t_{i+1}))-\varphi(g(t_i))}{g(t_{i+1})-g(t_i)}\\
=&
-\sum_{i=k_j+1}^{k_{j+1}-1} \frac{1}{2} \varphi''(g(t_i))\, (g(t_{i+1})-g(t_i))+\frac{1}{6} \varphi'''(\gamma_i)\,(g(t_{i+1})-g(t_i))^2 \\
\xrightarrow[N\rightarrow \infty]{} & -\frac{1}{2} \int_{a_j+}^{a_{j+1}-} \varphi''(g(s)) \, dg(s) + O(l^{-2}) 
= -\frac{1}{2} \int_{g(a_j+)}^{g(a_{j+1}-)} \varphi''(s) \, ds + O(l^{-2}).
\end{align*}
Summation over $j$ leads to
\begin{align*}
&\sum_{j=1}^{l-1}\sum_{i=k_j+1}^{k_{j+1}-1} \varphi'(g(t_i)) -\frac{\varphi(g(t_{i+1}))-\varphi(g(t_i))}{g(t_{i+1})-g(t_i)} \\
\xrightarrow[N\rightarrow \infty]{}&
-\frac{1}{2}\sum_{j=1}^{l-1} \int_{g(a_j+)}^{g(a_{j+1}-)} \varphi''(s) \, ds + O(l^{-1})
=-\frac{1}{2}\int_0^1 \varphi''(s) \, ds+ \frac{1}{2}\sum_{j=2}^{l-1} \int_{g(a_j-)}^{g(a_{j}+)} \varphi''(s) \, ds + O(l^{-1}) \\
\xrightarrow[l\rightarrow \infty]{}&-\frac{1}{2} (\varphi'(1)-\varphi'(0))+\frac{1}{2} \sum_{a\in J_g} \varphi'(g(a+))-\varphi'(g(a-)).
\end{align*}
Combining this with \eqref{sum_k_i} yields that the second line of \eqref{V_n_rewritten} converges in fact to \eqref{target_limit}, which completes the proof. \bbox
\end{proof}
Since $w_N(g(t_1),\ldots, g(t_{N-1})$ converges to $w$ in $L^p(\G,\Qbeta)$, $p>0$ (cf.\ proof of corollary \ref{simplifycorr} above), the last lemma ensures that the first term of $d^N\zeta$ converges to the first term of $\div_\Qbeta{\zeta}$ in $H$, while the following lemma deals with the second term.

\begin{lemma}
For $\Qbeta$-a.s.\ $g$ we have
\begin{align*}
\langle \nabla w_N(g(t_1),\ldots, g(t_{N-1})),\vec{\varphi}(g(t_1),\ldots, g(t_{N-1})) \rangle_{\mathbb{R}^{N-1}} \rightarrow \langle \nabla w_{|g},\varphi(g(.))\rangle_{L^2(0,1)}   , \quad \text{as } N \rightarrow \infty,
\end{align*}
and we have also convergence in $H$.

\end{lemma}

\begin{proof}
As in the proof of the last lemma it is enough to prove convergence $\Qbeta$-a.s. Note that
\begin{align*}
\langle \nabla  w_N(\vec g ),\vec{\varphi}(\vec g ) \rangle_{\mathbb{R}^{N-1}} =N \langle \iota^N(\nabla  w_N(\vec g )),\iota^N(\vec{\varphi}(\vec g ))\rangle_{L^2(0,1)} ,  
\end{align*}
writing $\vec g :=(g(t_1),\ldots, g(t_{N-1}))$ and using the extension of $\iota^N$ on $\mathbb{R}^{N-1}$.
By triangle and Cauchy-Schwarz inequality we obtain
\begin{align*}
&|\langle N \iota^N(\nabla  w_N(\vec g )),\iota^N(\vec{\varphi}(\vec g ))\rangle_{L^2(0,1)}-
\langle \nabla w_{|g},\varphi(g(.))\rangle_{L^2(0,1)}| \\
\leq &|\langle N \iota^N(\nabla  w_N(\vec g ))-\nabla w_{|g},\iota^N(\vec{\varphi}(\vec g ))\rangle_{L^2(0,1)}|
+\langle \nabla w_{|g},\iota^N(\vec{\varphi}(\vec g ))-\varphi(g(.))\rangle_{L^2(0,1)}|  \\
\leq & \norm{N \iota^N(\nabla  w_N(\vec g ))-\nabla w_{|g}}_{L^2(0,1)} \, \norm{\iota^N(\vec{\varphi}(\vec g ))}_{L^2(0,1)}
+\norm{\nabla w_{|g}}_{L^2(0,1)} \, \norm{\iota^N(\vec{\varphi}(\vec g ))-\varphi(g(.))}_{L^2(0,1)},
\end{align*}
which tends to zero by remark \ref{rem_Mosco2} and by the definition of $\iota^N$. \bbox
\end{proof}

\subsubsection{Proof of proposition \ref{identlimit}}
\begin{lemma} \label{iotacont}
\textit{For $u \in C(\G)$ let $u_N  \in H^N$ be defined by $u_N (x):= u(\iota x)$, then $u_N \to u$ strongly. Moreover, for any sequence $f_N \in H^N$ with $f_N \to f \in H$ strongly, $u_N \cdot f_N \to u\cdot f$ strongly.}
\end{lemma}

\begin{proof} Let $\tilde u_N \in H$ be defined by $\tilde u_N(g):= u(g^N)$, where $g^N:= \sum_{i=1}^N g(i/N)\one_{[i/N, (i+1)/N)}$, then $\tilde u_N \to u$ in $H$ strongly.  Moreover,
\[
\lim_N \lim_M \norm{ \Phi^M \tilde u_N - u_M}_{H^M} = \lim_N \lim_M \norm{ \Phi^M \tilde u_N - \tilde u_M}_{H}
 = \lim_N \norm{  \tilde u_N - u}_{H} = 0,
\]
where as above we have identified $\Phi^M$ with the corresponding projection operator in $L^2(\G,\Qbeta)$. For the proof of the second statement, let  $H \ni \tilde  f_N \to f$ in $H$  such that  $\lim_N \limsup_M  \norm {\Phi^M \tilde f_N -f_M}_{H^M} =0$. From the uniform boundedness of $\tilde u_N$ it follows that also $\tilde u_N \cdot \tilde f_N \to u \cdot f$ in $H$. In order to show  $H^M \ni u_M \cdot f_M \to u\cdot f$ write
\begin{align*}
  \norm {\Phi^M (\tilde u_N\cdot \tilde f_N) -u_M \cdot f_M}_{H^M} & \leq 
 \norm {\Phi^M (\tilde u_N\cdot \tilde f_N) -u_M \cdot \Phi^M (\tilde f_M)}_{H^M} +\norm {u_M \cdot f_M -u_M \cdot \Phi^M (\tilde f_M)}_{H^M}.
\end{align*}
Identifying the map $\Phi^M$ with the associated conditional expectation operator, considered as an orthogonal projection  in $H$, the claim follows from  
\begin{align*}
 \norm {\Phi^M (\tilde u_N\cdot \tilde f_N) -u_M \cdot \Phi^M (\tilde f_M)}_{H^M} &  = 
\norm { \Phi^M (\tilde u_N\cdot \tilde f_N)   -\tilde u_M \cdot  \Phi^M (\tilde f_M)}_{H} \\
& = 
\norm { \Phi^M (\tilde u_N\cdot \tilde f_N)   - \Phi^M ( \tilde u_M \cdot \tilde f_M)}_{H}  \\
& \leq 
\norm {  \tilde u_N\cdot \tilde f_N   -   \tilde u_M \cdot \tilde f_M}_{H}  \intertext{and} 
  \norm {u_M \cdot f_M -u_M \cdot \Phi^M (\tilde f_M)}_{H^M} &  \leq     \norm{u}_{\infty} \norm {  f_M - \Phi^M (\tilde f_N)}_{H^M} \nonumber \\
 & ~~~ +    \norm{u}_{\infty} \norm { \Phi^M (\tilde f_N)  - \Phi^M (\tilde f_M)}_{H^M} \nonumber
\\
&= \norm{u}_{\infty} \norm {  f_M - \Phi^M (\tilde f_N)}_{H^M} \nonumber \\
&~~~ +   \norm{u}_{\infty} \norm { \Phi^M (\tilde f_N)  - \Phi^M (\tilde f_M)}_{H} \nonumber
\\
& \leq \norm{u}_{\infty} \norm {  f_M - \Phi^M (\tilde f_N)}_{H^M} \nonumber \\
& ~~~ +   \norm{u}_{\infty} \norm {  \tilde f_N  -  \tilde f_M}_{H} \nonumber
\end{align*}
such that in fact $\lim_N \limsup_M \norm {\Phi^M (\tilde u_N\cdot \tilde f_N) -u_M \cdot f_M}_{H^M} =0$. \bbox 
\end{proof}

\textit{Proof of proposition \ref{identlimit}} It suffices to prove the claim for functions $f \in C(\G^l)$ of the form $f(g_1,\dots , g_l)= f_1(g_1)\cdot f_2(g_2) \dots \cdot f_l(g_l)$ with $f_i \in C(\G)$. Let  $P_t^N : H^N \to H^N$ be the semigroup on $H^N$ induced by $g^N$ via 
 $\mathbb E_{g\cdot q_N} [f(g_t^N)] = \langle P_t^N f, g\rangle_{H^N}$. From theorem \ref{mconvergence} and the abstract results in \cite{MR2015170} it follows that $P^N_t$ converges to $P_t$ strongly, i.e. for any sequence $u^N \in H^N$ converging to some $ u\in H$   strongly, the sequence $P^N_t u^N $ also strongly converges to $P_tu $. 
Let $f^N_i := f_i \circ \iota ^N$, then inductive application of  lemma \ref{iotacont}
yields 
\begin{align*}
P_{t_l-t_{l-1}}^N (f_l^N &\cdot P_{t_{l-1}-t_{l-2}}^N ( f^N_{l-1}\cdot  P_{t_{l-2}-t_{l-3}}^N\dots  f^N_2\cdot P_{t_1}^N f^N_1)\cdots) \\
&\stackrel{N \to \infty}{\longrightarrow}  P_{t_l-t_{l-1}}  (f_l  \cdot P_{t_{l-1}-t_{l-2}}  ( f _{l-1}\cdot  P_{t_{l-2}-t_{l-3}}\dots  f_2\cdot P_{t_1} f_1)\cdots) \mbox{  strongly},  
\end{align*} 
which in particular implies the convergence of inner products. Hence, using the Markov property of $g^N$ and $g$ we may conclude that
\begin{align}
 \lim_N  \mathbb E\bigl( f_1(g^N_{t_1}) & \dots f_l(g^N_{t_l})\bigr)  = \lim_N \mathbb E\bigl( f_1^N(X^N_{t_1}) \dots f_l^N(X^N_{t_l})\bigr) \nonumber \\
& = \lim_N \langle  1, P_{t_l-t_{l-1}}^N (f_l^N \cdot P_{t_{l-1}-t_{l-2}}^N ( f^N_{l-1}\cdot  P_{t_{l-2}-t_{l-3}}^N\dots  f^N_2\cdot P_{t_1}^N f^N_1)\cdots)\rangle_{H^N} \nonumber \\
& =  \langle  1, P_{t_l-t_{l-1}} (f_l \cdot P_{t_{l-1}-t_{l-2}} ( f_{l-1}\cdot  P_{t_{l-2}-t_{l-3}}\dots  f_2\cdot P_{t_1} f_1)\cdots)\rangle_{H}\nonumber  \\
& = \mathbb E\bigl( f_1(g_{t_1})  \dots f_l(g_{t_l})\bigr). \tag*{$\Box$}
\end{align}

\section{Appendix: On a connection to $\nabla \phi$-interface models}
We conclude with a remark on  a link  to stochastic interface models.  \nocite{MR2228384} Consider an interface on the one-dimensional lattice $\Gamma_N:=\{1,\ldots,N-1\}$. The location of the interface at time $t$ is represented by the height variables $\phi_t=\{\phi_t(x),\, x\in \Gamma_N\}\in \sqrt{N} \cdot  \Sigma_N$ with dynamics determined by the generator $\tilde{L}^N$ defined below and with the boundary conditions $\phi_t(0)=0$ and $\phi(N)=\sqrt{N}$ at $\partial \Gamma_N:=\{0,N\}$.
\[ \tilde{L}^N f(\phi):= (\frac \beta N -1 )   \sum_{x\in \Gamma_N}\left ( \frac 1 {\phi(x) - \phi(x-1)} - \frac 1 {\phi(x+1) - \phi(x)} \right) \frac {\partial }{\partial \phi(x)} f(\phi) +  \Delta f(\phi)\] for $\phi\in \mbox{Int}(\sqrt{N}\cdot\Sigma_N)$ and with $\phi(0):=0 $ and  $\phi(N):=\sqrt{N}$.  $\tilde{L}^N$ corresponds  to $L^N$ as an operator on $C(\sqrt{N} \cdot\Sigma_N)$ with domain
\[ \D(\tilde{L}^N) := \{ f\in C^2(\sqrt{N} \cdot \Sigma_N)\,|\, \tilde{L}^N f \in C(\sqrt{N} \cdot\Sigma_N)\}. \] 
Note that this  system involves a non-convex interaction potential function $V$ on $(0,\infty)$ given by $V(r)=(1- \frac \beta N) \log(r)$ and the Hamiltonian
\[
H_N(\phi):=\sum_{x=0}^{N-1} V(\phi(x+1)-\phi(x)), \qquad \phi(0):=0, \, \phi(N):=\sqrt{N}.
\]
Then, the natural stationary distribution of the interface is the Gibbs measure $\mu_N$ conditioned on $\sqrt{N} \cdot \Sigma_N$:
\[
\mu_N(d\phi):=\frac 1 {Z_N} \exp(-H_N(\phi)) \one_{\{(\phi(1),\ldots, \phi(N-1))\in\sqrt{N}\cdot\Sigma_N\}} \prod_{x\in \Gamma_N} d\phi(x)
,\]
where $Z_N$ is a normalization constant. Note that $\mu_N$ is the corresponding measure of $q_N$ on the state space $\sqrt{N}\cdot \Sigma_N$.
Suppose now that $(\phi_t)_{t\geq 0}$ is the stationary process generated by $\tilde{L}^N$. Then the space-time scaled process
\[
\tilde{\Phi}^N_t (x):=\frac{1}{\sqrt{N}} \phi_{N^2 t}(x), \qquad x=0,\ldots,N,
\] 
living on $\Sigma_N$ is associated with the Dirichlet form $N\cdot \E^N$. Introducing the $\G$-valued fluctuation field
\[
\Phi^N_t (\vartheta):=\sum_{x\in \Gamma_N} \tilde{\Phi}^N_t (x) \, \one_{[x/N,(x+1)/N)} (\vartheta),  \qquad \vartheta \in [0,1),
\]
 by our main result we have weak convergence for the law of the equilibrium fluctuation field $\Phi^N$ to the law of the nonlinear diffusion on $\G$, which is the $\G$-parametrization of the Wasserstein diffusion.\\

{\bf Acknowledgements:} We thank Michael R\"ockner for explaining the importance of Markov uniqueness to us during the 2007 German-Japanese conference on Stochastic Analysis in Berlin. Many thanks go also to Theresa Heeg for providing us with the results of her simulation studies.


\bibliographystyle{abbrv}
\def\cprime{$'$}

\end{document}